\documentclass{amsart} 
\usepackage{amsthm, amsmath, mathtools}
\usepackage{amssymb}
\usepackage{amscd}
\usepackage{enumerate}
\usepackage{mathrsfs} 
\usepackage[curve,matrix,arrow,frame,tips]{xy}
\usepackage{colonequals}
\usepackage{verbatim}
\usepackage{pdfsync}
\usepackage{graphicx}
\usepackage[usenames,dvipsnames]{color}

\usepackage{url}

\usepackage{fullpage}

\usepackage[normalem]{ulem}

\numberwithin{equation}{section} 

\theoremstyle{plain}
\newtheorem{theorem}{Theorem}[section]

\newtheorem{lemma}[theorem]{Lemma}

\newtheorem{corollary}[theorem]{Corollary}

\newtheorem{question}[theorem]{Question}

\theoremstyle{definition}
\newtheorem{definition}[theorem]{Definition}

\theoremstyle{remark}
\newtheorem{remark}[theorem]{Remark}

\newcommand{\Q}{\mathbb{Q}}

\newcommand{\Gm}{{\mathbb{G}_m}}
\newcommand{\Z}{\mathbb{Z}}

\DeclareMathOperator{\Spec}{Spec}
\DeclareMathOperator{\Hom}{Hom}

\DeclareMathOperator{\GL}{GL}

\DeclareMathOperator{\Char}{char}
\DeclareMathOperator{\Gal}{Gal}

\DeclareMathOperator{\Pic}{Pic}
\DeclareMathOperator{\Br}{Br}
\DeclareMathOperator{\GSp}{GSp}

\newcommand{\PP}{\mathbb{P}}

\renewcommand{\to}{\longrightarrow}

\DeclareMathOperator{\Heis}{\mathcal{H}}

\newcommand{\et}{\mathrm{\acute{e}t}}

\usepackage{microtype}

\begin{document}

\title{Unramified Heisenberg group extensions of number fields }

\author[F. Bleher]{Frauke M. Bleher}
\address{F.B.: Department of Mathematics\\University of Iowa\\
14 MacLean Hall\\Iowa City, IA 52242-1419\\ U.S.A.}
\email{frauke-bleher@uiowa.edu}
\thanks{The first author was supported in part by NSF  Grant No.\ DMS-1801328.}

\author[T. Chinburg]{Ted Chinburg}
\address{T.C.: Department of Mathematics\\University of Pennsylvania\\
Philadelphia, PA 19104-6395\\ U.S.A.}
\email{ted@math.upenn.edu}
\thanks{The second author is the corresponding author. He was supported in part by NSF FRG Grant No.\ DMS-1360767, NSF SaTC grants No. CNS-1513671 and No. CNS-1701785.}

\author[J. Gillibert]{Jean Gillibert}
\address{J. G.:  Institut de Math{\'e}matiques de Toulouse \\ CNRS UMR 5219 \\
118, route de Narbonne, 31062 Toulouse Cedex \\ France.}
\email{Jean.Gillibert@math.univ-toulouse.fr}

\date{\today}


\begin{abstract}We construct \'etale  generalized Heisenberg group covers of hyperelliptic curves over number fields.  We use these to produce  infinite families of quadratic extensions of cyclotomic fields that admit everywhere unramified generalized Heisenberg Galois extensions.
\end{abstract}

\maketitle


\section{Introduction}
\label{s:intro}

Let $G$ be a finite group and suppose $K$ is a number field.  
In this paper we consider the problem of constructing infinitely many unramified Galois $G$-extensions $M/L$ of number fields 
for which $L$ has bounded degree over  $K$.  When one imposes no conditions on the ramification of $M/L$, a classical technique
is to specialize a $G$-cover  $\pi:C' \to C$ of curves over  $K$ for which $C$ has many points over such $L$.  To construct $M/L$ that are unramified,
a first step is to ensure that $\pi$ itself is unramified.  One can then
find a model of $\pi$ over a ring of $S$-integers of $K$ that is unramified, and control the ramification of $M/L$ over places in $S$ by
imposing conditions on how one specializes $\pi$.  This approach was used in \cite{bg18} by Bilu and the third author (see also \cite{Lev08}) to construct abelian unramified extensions of quadratic extensions
of a given number field, using abelian covers $\pi$ of a hyperelliptic curve $C$. 

 In this
paper we consider  nilpotent groups $G$.   The essential obstruction to the above technique is the construction of an unramified $G$-cover
$\pi:C' \to C$ over $K$.  When $G$ is abelian, one constructs such $\pi$ by imposing conditions on the $K$-rational torsion points of the Jacobian of $C$.
For more general nilpotent covers, there is a  technique which combines abelian information with data on cup products and higher Massey products
of characters of Galois groups;  see  \cite{Sharifi, SharifiCrelle} and their references, for example.   However, the calculation of cup products is 
more difficult when $K$ is not algebraically closed.  This is due to the fact that $H^2(C,\mu_n)$ contains a subgroup isomorphic to $\mathrm{Pic}^0(C)/n\mathrm{Pic}^0(C)$ when $\mu_n$ is the \'etale sheaf of $n$-{th} roots of unity, and this subgroup is in general non-trivial.  

 The main innovation in this paper is to exploit the  action of $\mathrm{Aut}_K(C)$ on cup products in order to construct unramified nilpotent covers $\pi:C' \to C$.
We will illustrate this by taking $G$ to be  a (generalized) Heisenberg group of the form $\Heis_{2d+1}(\Z/n\Z)$; see Section~\ref{s:Hgroups} for a definition of this group. 
It is 
non-abelian and nilpotent of order $n^{2d+1}$ and it is contained in the subgroup of unipotent upper-triangular matrices in $\GL_{d+2}(\Z/n\Z)$. For example, we will produce
by the above methods  some infinite families of everywhere unramified $\Heis_{3}(\Z/n\Z)$-extensions $M/L$ with $L$ a quadratic extension of $\mathbb{Q}(\zeta_n)$ (cf. Theorem~\ref{thm:intro}). This method differs from other approaches based on group theory, Hurwitz spaces and the Inverse Galois problem which we will describe in Remark \ref{rem:priorresults}.  Previous work by V\"olklein and others (see \cite{Volk1} and \cite{Volk2}) leads to alternate proofs of some of the results in this paper, sometimes with stronger hypotheses, e.g. that $n$ is prime.

Let $C$ be a smooth projective hyperelliptic curve over a number field $K$.  Our first result, Theorem~\ref{thm:main}, shows that the existence of an {\'e}tale $\Heis_{2d+1}(\Z/n\Z)$-cover of  $C$ is equivalent to the existence of two families (each one having $d$ elements) of $n$-torsion line bundles on $C$, which are globally orthogonal under the Weil pairing. We state Theorem~\ref{thm:main} in the setting of twisted Heisenberg group schemes, in which $\Z/n\Z$ is replaced by $\mu_n$;  this allows us to work over smaller fields in the absence of roots of unity.
As an illustration of Theorem~\ref{thm:main}, we give at the end of Section~\ref{s:GeoH} several explicit examples of Heisenberg Galois covers of hyperelliptic curves.
In this case, the strategy of applying elements of $\mathrm{Aut}_K(C)$ to control cup products amounts to using the hyperelliptic involution of $C$ to
show that the relevant cup product over $K$ is trivial provided its base change to $\overline{K}$ is trivial.

We will apply the specialization results of \cite{bg18} to the examples of covers constructed in Section~\ref{s:GeoH}.  Suppose one is given a connected {\'e}tale cover $\pi:C' \to C$ over a number field $K$ in which some $K$-rational point $P_0$ splits completely.  In \cite{bg18},  it was shown that one can find infinitely many (in a strong quantitative sense) points in $C(\overline{K})$  such that the extension of number fields that results on specializing $\pi$ over these points has the same degree as $\pi$ and is everywhere unramified. This leads to  the following result.

\begin{theorem}
\label{thm:intro}
Let $n>1$ be an odd integer, and let $\zeta_n$ be a primitive $n$-th root of unity. Then there exist infinitely many quadratic extensions $L/\Q(\zeta_n)$ which admit a Galois extension  with group $\Heis_3(\Z/n\Z)$, unramified at all finite places of $L$. Moreover, given a finite set $S$ of places of $\Q(\zeta_n)$, which may contain finite and infinite places, there exist infinitely many such $L$ for which primes lying above $S$ in $L$ are totally split in the corresponding Galois extension.

Furthermore, there is a constant $c>0$, which depends only on $n$ and $S$, for which the following is true.  For sufficiently large positive $N$, the number of (isomorphism classes of) such fields $L$ whose relative discriminant $\Delta(L/\Q(\zeta_n))$ satisfies
$$
\bigl|\mathrm{N}_{\Q(\zeta_n)/\Q}\Delta(L/\Q(\zeta_n))\bigr|^{1/\varphi(n)} \leq N
$$
is at least $cN^{\varphi(n)/(4n-2)}/\log N$, where $\varphi$ is Euler's function.
\end{theorem}

Two potential generalizations of our results are (i) to replace hyperelliptic curves by more general ones, and (ii) to replace Heisenberg
groups by more general nilpotent groups.  These generalizations are connected because to treat nilpotent groups with higher nilpotency, one must consider higher Massey products.  To control such products with automorphisms of the base curve will require using more than hyperelliptic involutions.

\medbreak

\noindent {\bf Acknowledgements.} The first and second authors would like to thank the University of Toulouse for its support and hospitality during work on this paper.   
The authors would also like to thank the referee for many helpful suggestions. 


\section{Heisenberg groups}
\label{s:Hgroups}

In this section, we fix two integers $n>1$ and $d\geq 1$. The Heisenberg group of rank $2d+1$ with coefficients in $\Z/n\Z$, denoted by $\Heis_{2d+1}(\Z/n\Z)$, is the subgroup of $\GL_{d+2}(\Z/n\Z)$ consisting of matrices of the form
$$
\begin{pmatrix}
 1 & \mathbf{a} & c\\
 0 & I_d & \mathbf{b}\\
 0 & 0 & 1\\
\end{pmatrix}
$$
where $I_d$ is the $d\times d$ identity matrix, $\mathbf{a}$ (resp. $\mathbf{b}$) is a row (resp. column) vector of length $d$ with coefficients in $\Z/n\Z$, and $c$ belongs to $\Z/n\Z$. The center of this group is the set of matrices satisfying $\mathbf{a}=0$ and $\mathbf{b}=0$, which is isomorphic to $\Z/n\Z$. It follows that we have an exact sequence of groups
\begin{equation}
\label{eq:HCE}
\begin{CD}
0 @>>> \Z/n\Z @>>> \Heis_{2d+1}(\Z/n\Z) @>>> (\Z/n\Z)^{2d} @>>> 0. \\
\end{CD}
\end{equation}
Thus $\Heis_{2d+1}(\Z/n\Z)$ is a central extension of $(\Z/n\Z)^{2d}$ by $\Z/n\Z$.

The twisted Heisenberg group scheme of rank $2d+1$ over $\Z[\frac{1}{n}]$, denoted by $\Heis_{2d+1}(\mu_n)$, is defined in the same way as the Heisenberg group $\Heis_{2d+1}(\Z/n\Z)$, but the vectors $\mathbf{a}$ and $\mathbf{b}$ in the matrix have coefficients in $\mu_n$, and $c$ belongs to $\mu_n^{\otimes 2}$. We have as previously an exact sequence
\begin{equation}
\label{eq:HCEtwist}
\begin{CD}
0 @>>> \mu_n^{\otimes 2} @>>> \Heis_{2d+1}(\mu_n) @>>> (\mu_n)^{2d} @>>> 0 \\
\end{CD}
\end{equation}
which is in fact an exact sequence of presheaves.

We note that $\Heis_{2d+1}(\mu_n)$ is a finite {\'e}tale $\Z[\frac{1}{n}]$-group scheme. We underline the fact that $\mu_n^{\otimes 2}$ is not representable by a finite flat group scheme over $\Z$, so $\Heis_{2d+1}(\mu_n)$ does not extend to a finite flat group scheme over $\Z$. Nevertheless, if $\zeta_n$ denotes a primitive $n$-th root of unity, then we have a non-canonical isomorphism between $\Heis_{2d+1}(\mu_n)$ and $\Heis_{2d+1}(\Z/n\Z)$ over $\Z[\frac{1}{n},\zeta_n]$.  Hence $\Heis_{2d+1}(\mu_n)$ extends to a constant group scheme over  $\Z[\zeta_n]$.

Unless otherwise specified, all torsors we consider are relative to the {\'e}tale topology. Thus, if $\Gamma$ is an \'etale group scheme over a scheme $X$, we denote by $H^1(X,\Gamma)$ the pointed set of isomorphism classes of $\Gamma$-torsors over $X$ for the {\'e}tale topology. When $\Gamma$ is abelian, this set is an abelian group.

If $\phi:\Gamma\to \Lambda$ is a morphism of $X$-group schemes, and if $\xi$ is a $\Gamma$-torsor over $X$, then the image of $\xi$ by the natural map $\phi_*:H^1(X,\Gamma)\to H^1(X,\Lambda)$ is a $\Lambda$-torsor.  We say, by abuse of notation, that this is the $\Lambda$-torsor associated to $\xi$, the morphism $\phi$ being omitted.

Let us recall a result of Sharifi  \cite{Sharifi} which allows one to produce torsors for $\Heis_{2d+1}(\mu_n)$ whose associated $(\mu_n)^{2d}$-torsor is given.

\begin{theorem}
\label{thm:Sharifi}
Let $X$ be a connected $\Z[\frac{1}{n}]$-scheme. Suppose we are given two $d$-tuples $\chi_1,\dots,\chi_d$ and $\chi_1',\dots,\chi_d'$ in $H^1(X,\mu_n)$ for some integer $d\geq 1$. Then there exists a $\Heis_{2d+1}(\mu_n)$-torsor $\xi$ over $X$ whose associated $(\mu_n)^{2d}$-torsor is the $2d$-tuple $(\chi_1,\dots,\chi_d,\chi_1',\dots, \chi_d')$ if and only if
$$
\sum_{i=1}^d [\chi_i\cup \chi_i'] = 0
$$
where the sum is computed in the group $H^2(X,\mu_n^{\otimes 2})$. Moreover, $\xi$ is connected if and only if the subgroup of $H^1(X,\mu_n)$ generated by $\chi_1,\dots,\chi_d,\chi_1',\dots, \chi_d'$ is isomorphic to $(\Z/n\Z)^{2d}$.
\end{theorem}

\begin{proof}
See \cite[Proposition 2.3]{Sharifi}. We note that Sharifi's result is stated in terms of (twisted) Galois representations over a field $K$ of characteristic prime to $n$, instead of torsors over a $\Z[\frac{1}{n}]$-scheme $X$. Nevertheless, his results immediately extend to our setting by considering representations of the {\'e}tale fundamental group of $X$.
\end{proof}

\begin{remark}
If $p$ and $q$ are relatively prime integers, it follows from the Chinese remainder theorem that $\Heis_{2d+1}(\Z/pq\Z)\simeq \Heis_{2d+1}(\Z/p\Z)\times \Heis_{2d+1}(\Z/q\Z)$, and similarly for twisted Heisenberg group schemes. Thus, we could assume that $n$ is the power of some prime number. However, we prefer for simplicity to work with arbitrary $n$.
\end{remark}


\section{Geometric Heisenberg group extensions}
\label{s:GeoH}

Let $K$ be a field, and let $\overline{K}$ be an algebraic closure of $K$. In our terminology, a hyperelliptic curve over $K$ is a smooth projective geometrically connected $K$-curve of genus $g\geq 1$, endowed with a degree $2$ map $\pi:C\to \PP^1_K$. This includes elliptic curves over $K$. In this setting, the Weierstrass points of $C$ are none other than the ramification points of $\pi$. We denote by $\tau$  the hyperelliptic involution of $C$.  The group $G = \{e,\tau\}$ acts on $C$, and the quotient morphism is the map $\pi:C \to \mathbb{P}^1_K$.

If $\xi$ is a $\Gamma$-torsor over $C$, and if $P_0 \in C(K)$ is a $K$-rational point of $C$, we say that $\xi$ splits over $P_0$ if $P_0^*\xi$ is the trivial $\Gamma$-torsor over $K$.

\begin{theorem}
\label{thm:main}
Let $n>1$ be an odd integer with $\Char(K)\nmid n$, and such that $[K(\mu_n):K]$ is prime to $n$. Let $C$ be a hyperelliptic curve over $K$, together with a $K$-rational Weierstrass point $P_0 \in C(K)$. Let $L_1,\dots,L_d$ and $L_1',\dots,L_d'$ in $\Pic^0(C)[n]$ such that
$$
\prod_{i=1}^d e_n(L_i,L_i')=1
$$
where $e_n$ denotes the Weil pairing. Then:
\begin{enumerate}
\item For $i=1,\dots,d$, there exists a unique \'etale $\mu_n$-torsor $\chi_i$ (resp. $\chi_i'$) over $C$ which splits over $P_0$, and whose associated $\Gm$-torsor is $L_i$ (resp. $L_i'$).
\item There exists an {\'e}tale $\Heis_{2d+1}(\mu_n)$-torsor $\xi$ which splits over the point $P_0$, and whose associated $(\mu_n)^{2d}$-torsor is the $2d$-tuple $(\chi_1,\dots,\chi_d,\chi_1',\dots, \chi_d')$. 
\item The torsor $\xi$ is geometrically connected if and only if the subgroup generated by the $L_i$ and the $L_i'$ is isomorphic to $(\Z/n\Z)^{2d}$.
\end{enumerate}
\end{theorem}

\begin{remark}
The Weil pairing is a non-degenerate bilinear pairing of finite $K$-group schemes
$$
e_n:J_C[n]\times J_C[n]\to \mu_n
$$
where $J_C$ denotes the Jacobian of $C$. In particular, if $L$ and $L'$ belong to $\Pic^0(C)[n]=J_C[n](K)$, then $e_n(L,L')$ belongs to $\mu_n(K)$.
\end{remark}

\begin{remark}
In Theorem~\ref{thm:main}~(2), the $\Heis_{2d+1}(\mu_n)$-torsor $\xi$ is unique up to a twist by a $\mu_n^{\otimes 2}$-torsor over $C$, which splits over $P_0$. This can be checked by going through the proof of Lemma~\ref{lem:centraltwist} below.
\end{remark}

\begin{remark}
\label{rem:priorresults}  The proof we will give of Theorem~\ref{thm:main} uses the hyperelliptic involution of $C$ to control cup products.  We thank the referee for
outlining a different approach which uses group theory, Hurwitz spaces and work on the Inverse Galois problem under some additional hypotheses.
For simplicity, assume $K(\mu_n) = K$, so that $\mu_n$ is isomorphic to $\mathbb{Z}/n\mathbb{Z}$ and $\Heis_{2d+1}(\mu_n)$ is isomorphic to the constant group
$\Heis_{2d+1}(\mathbb{Z}/n\mathbb{Z})$. Theorem \ref{thm:main}  follows if one can construct a regular cover of $\mathbb{P}^1_K$ with group $\Gamma = \Heis_{2d+1}(\mathbb{Z}/n\mathbb{Z})  \rtimes \mathbb{Z}/2\mathbb{Z}$ having $C$ as the quotient by $\Heis_{2d+1}(\mathbb{Z}/n\mathbb{Z}) $
and having inertia groups that are involutions mapping to the hyperelliptic involution of $C$.  Such a $\Gamma$-cover is a central $\mathbb{Z}/n\mathbb{Z}$-extension of a cover with group $G = (\mathbb{Z}/n\mathbb{Z})^{2d} \rtimes \mathbb{Z}/2\mathbb{Z}$.  When $n$ is prime,  \cite[Theorem 9.17.(1)]{Volk2} describes a method of
constructing $G$-covers of $\mathbb{P}^1_K$ of the required kind when $C$ is allowed to vary using Hurwitz spaces.  One can then apply results from
\cite{Volk1} to show that the constrained central embedding problem associated to constructing an appropriate $\Gamma$-cover has a solution under appropriate hypotheses.
\end{remark}

To prove Theorem~\ref{thm:main} we need the following results.

\begin{lemma}
\label{lem:ginvariant}  For all integers $j \ge 0$ there is a canonical isomorphism $H^j(\mathbb{P}^1_K,\mu_n) = H^j(C,\mu_n)^G$.
\end{lemma}

\begin{proof} This is clear from the spectral sequence $H^p(G,H^q(C,\mu_n)) \Rightarrow H^{p+q}(\mathbb{P}^1_K,\mu_n)$ together with the fact that $G$ has order $2$ and all of the groups $H^q(C,\mu_n)$ are annihilated by multiplication by the odd integer $n$.
\end{proof}

\begin{lemma}
\label{lem:P1case}
There are canonical isomorphisms
\begin{equation}
\label{eq:Kummer1}
H^1(\mathbb{P}^1_K,\mu_n) = K^*/(K^*)^n\quad \mathrm{and}\quad H^1(\mathbb{P}^1_K,\mathbb{G}_m) = \mathrm{Pic}(\mathbb{P}^1_K) =\mathbb{Z}
\end{equation}
and an isomorphism of Brauer groups 
\begin{equation}
\label{eq:BRP1}
H^2(\mathbb{P}^1_K,\mathbb{G}_m) = \mathrm{Br}(\mathbb{P}^1_K) = \mathrm{Br}(K)
\end{equation}
induced by pulling back Azumaya algebras from $K$ to $\mathbb{P}^1_K$.  There is an 
exact sequence
\begin{equation}
\label{eq:nice}
0 \to \mathrm{Pic}(\mathbb{P}^1_K)/n \to H^2(\mathbb{P}^1_K,\mu_n) \to H^2(\mathbb{P}^1_K,\mathbb{G}_m)[n] \to 0.
\end{equation}
Using the above isomorphisms, the sequence (\ref{eq:nice}) becomes
\begin{equation}
\label{eq:nicer}
0 \to \mathbb{Z}/n\mathbb{Z} \to H^2(\mathbb{P}^1_K,\mu_n) \to  \Br(K)[n] \to 0.
\end{equation}
Moreover, any class in $H^2(\mathbb{P}^1_K,\mu_n)$ that splits over some $K$-rational point of $\mathbb{P}^1_K$ belongs to the kernel of the homomorphism $H^2(\mathbb{P}^1_K,\mu_n) \to \Br(K)[n]$.
\end{lemma}

\begin{proof} One has $H^1(\mathbb{P}^1_K,\mathbb{G}_m) = \mathrm{Pic}(\mathbb{P}^1_K) = \mathbb{Z}$
via the degree map.  
The cohomology of the Kummer sequence
$$1 \to \mu_n \to \mathbb{G}_m \xrightarrow{[n]} \mathbb{G}_m \to 1$$
then gives (\ref{eq:Kummer1}) and (\ref{eq:nice}).  To analyze $H^2(\mathbb{P}^1_K,\mathbb{G}_m)$
we use the Hochschild-Serre spectral sequence
\begin{equation}
\label{eq:HS}
H^p(K,H^q(\mathbb{P}^1_{\overline{K}},\mathbb{G}_m)) \Rightarrow H^{p+q}(\mathbb{P}^1_K,\mathbb{G}_m).
\end{equation}
By Tsen's theorem, $H^2(\mathbb{P}^1_{\overline{K}},\mathbb{G}_m) = 0$.  The action 
of the profinite group $\mathrm{Gal}(\overline{K}/K)$ on $H^1(\mathbb{P}^1_{\overline{K}},\mathbb{G}_m) = \mathrm{Pic}(\mathbb{P}^1_{\overline{K}}) = \mathbb{Z}$ is trivial, so
$$H^1(K,H^1(\mathbb{P}^1_{\overline{K}},\mathbb{G}_m)) = 0.$$
We have
$$H^2(K,H^0(\mathbb{P}^1_{\overline{K}},\mathbb{G}_m)) = H^2(K,\overline{K}^*) = \Br(K).$$
Finally, the restriction map 
$$H^1(\mathbb{P}^1_K,\mathbb{G}_m) \to H^1(\mathbb{P}^1_{\overline{K}},\mathbb{G}_m)$$
is an isomorphism.   Putting these facts into the spectral sequence (\ref{eq:HS}) gives
(\ref{eq:BRP1}).  The last statement follows from the fact that the pullback from $K$ to $\mathbb{P}^1_K$ induces
a section of the surjection $H^2(\mathbb{P}^1_K,\mu_n) \to \mathrm{Br}(K)[n]$.
\end{proof}

The final lemma we will need to prove Theorem \ref{thm:main} has to do with twisting Heisenberg torsors in order to ensure that they split over
a particular point.

\begin{lemma}
\label{lem:centraltwist}
Let $\xi$ be a $\Heis_{2d+1}(\mu_n)$-torsor over $C$, whose associated $(\mu_n)^{2d}$-torsor splits over $P_0$. Then there exists a $\Heis_{2d+1}(\mu_n)$-torsor $\xi'$ over $C$ with the same associated $(\mu_n)^{2d}$-torsor such that $\xi'$ splits over $P_0$ and  $\xi'$ is isomorphic to $\xi$ over $\overline{K}$.
\end{lemma}

\begin{proof}
Consider the following commutative diagram of pointed sets with exact rows \cite[Chap. III, Proposition 3.3.1]{Giraud71}
$$
\begin{CD}
1 @>>> H^1(C,\mu_n^{\otimes 2}) @>>> H^1(C,\Heis_{2d+1}(\mu_n)) @>a>> H^1(C,(\mu_n)^{2d}) \\
@. @VP_0^*VV @VP_0^*VV @VP_0^*VV \\
1 @>>> H^1(K,\mu_n^{\otimes 2}) @>>> H^1(K,\Heis_{2d+1}(\mu_n)) @>a_K>> H^1(K,(\mu_n)^{2d}) \\
\end{CD}
$$
in which the vertical maps are the restrictions to $P_0$.
By definition of an exact sequence of pointed sets, the kernel of $a$ is exactly the image of the map $H^1(C,\mu_n^{\otimes 2})\to H^1(C,\Heis_{2d+1}(\mu_n))$, and similarly for $a_K$. 

By assumption, $a(\xi)$ belongs to the kernel of $P_0^*$, and hence $P_0^*\xi$ belongs to the kernel of $a_K$ by commutativity. It follows that $P_0^*\xi$ comes from a $\mu_n^{\otimes 2}$-torsor over $K$, which we denote by $c_0$. Let us denote by $f:C\to \Spec(K)$ the structural morphism. Since the group $\mu_n^{\otimes 2}$ is central in $\Heis_{2d+1}(\mu_n)$, it follows from \cite[Chap. III, Remarque 3.4.4]{Giraud71} that the contracted product
$$
\xi':=\xi\times^{\mu_n^{\otimes 2}}_C f^*c_0
$$
is a $\Heis_{2d+1}(\mu_n)$-torsor over $C$, which splits over $P_0$ because $P_0^*f^*c_0=c_0$. Finally, $\xi$ and $\xi'$ are isomorphic over $\overline{K}$, because $c_0$ is just a Galois cohomology class over $K$, hence splits over $\overline{K}$.
\end{proof}

\begin{proof}[Proof of Theorem \ref{thm:main}]
Since the curve $C$ is geometrically connected, we have $\Gm(C)=\Gm(K)=K^*$. Hence the Kummer exact sequence on $C$ gives 
$$
0\to K^*/(K^*)^n \to H^1(C,\mu_n) \to \Pic(C)[n] \to 0.
$$
Moreover, the map $P_0^*:H^1(C,\mu_n)\to H^1(K,\mu_n)\simeq K^*/(K^*)^n$ is a retraction of the natural map $K^*/(K^*)^n \to H^1(C,\mu_n)$. One deduces that, given $L\in\Pic(C)[n]$, there exists a unique $\mu_n$-torsor $\chi$ on $C$ which splits over the point $P_0$, and whose associated $\Gm$-torsor is $L$. This proves part (1).

For part (2), let us first prove that $\tau(\chi_i)=-\chi_i$ for all $i$. We have a sequence of canonical isomorphisms
$$
H^1(K,\mu_n) \simeq H^1(\mathbb{P}^1_K,\mu_n) \simeq H^1(C,\mu_n)^G
$$
and the map $P_0^*:H^1(C,\mu_n)^G \to H^1(K,\mu_n)$ is the inverse isomorphism. Now, the $\mu_n$-torsor $\chi_i+\tau(\chi_i)$ is invariant under the action of $G$, and $P_0$ (which is a ramification point of $C\to \mathbb{P}^1_K$) is invariant by $\tau$, from which it follows that
$$
P_0^*(\chi_i+\tau(\chi_i))=2P_0^*\chi_i = 0.
$$

The map $P_0^*$ being an isomorphism, this proves that $\chi_i+\tau(\chi_i)=0$. It follows that we have, for $i=1,\dots,d$,
$$
\tau(\chi_i\cup \chi_i') = \tau(\chi_i) \cup \tau(\chi_i') = (-\chi_i) \cup (-\chi_i') = \chi_i \cup \chi_i'.
$$
Thus $\chi_i\cup \chi_i'\in H^2(C,\mu_n^{\otimes 2})^G$.

The map $H^2(C,\mu_n^{\otimes 2})\to H^2(C\otimes_K K(\mu_n),\mu_n^{\otimes 2})$ is injective, because $[K(\mu_n):K]$ is prime to $n$. Hence, in order to prove that $\sum_{i=1}^d [\chi_i\cup \chi_i']=0$, we may (and we do) assume that $K$ contains a primitive $n$-th root of unity. Then for any $K$-scheme $X$ and for all $j$ we have natural isomorphisms
$$
H^j(X,\mu_n^{\otimes 2}) \simeq H^j(X,\mu_n) \otimes \mu_n.
$$

By Lemma~\ref{lem:P1case}, we have a  commutative diagram
\begin{equation}
\label{eq:bigdiagram}
\begin{CD}
0 @>>> (\Pic(\mathbb{P}^1_K)/n)\otimes \mu_n @>>> H^2(\mathbb{P}^1_K,\mu_n^{\otimes 2}) @>>> \Br(K)[n]\otimes \mu_n @>>> 0 \\
@. @VVV @VVV @VVV \\
0 @>>> (\Pic(C)/n)\otimes \mu_n @>>> H^2(C,\mu_n^{\otimes 2}) @>>> \Br(C)[n]\otimes \mu_n @>>> 0 \\
@. @VVV @VVV \\
@. (\Pic(C\otimes_K\overline{K})/n)\otimes \mu_n @>\sim>> H^2(C\otimes_K\overline{K},\mu_n^{\otimes 2}) \\
\end{CD}
\end{equation}
in which the rows are obtained by tensoring the Kummer exact sequences by $\mu_n$, and the vertical maps are obtained by base change.

We have proved that $\chi_i\cup \chi_i'$ is $G$-invariant, so that it comes from a class in $H^2(\mathbb{P}^1_K,\mu_n^{\otimes 2})$ by Lemma~\ref{lem:ginvariant}. Moreover, because $\chi_i\cup\chi_i'$ splits over the point $P_0$, the corresponding class in  $H^2(\mathbb{P}^1_K,\mu_n^{\otimes 2})$ splits over the image of $P_0$ in $\mathbb{P}^1_K$.   By diagram \eqref{eq:bigdiagram} and Lemma~\ref{lem:P1case}, it belongs to the image of $(\Pic(\mathbb{P}^1_K)/n)\otimes \mu_n \to H^2(\mathbb{P}^1_K,\mu_n^{\otimes 2})$.

On the other hand, the map $C\to \mathbb{P}^1_K$ has degree $2$, hence the natural map
$$
(\Pic(\mathbb{P}^1_K)/n) \to (\Pic(C\otimes_K\overline{K})/n)
$$
can be identified, via the degree map, with the multiplication-by-$2$ map $[2]:\Z/n\Z \to \Z/n\Z$,
which is an isomorphism, since $n$ is odd. It follows that the composition of the two vertical maps in the first column of \eqref{eq:bigdiagram} is the multiplication-by-$2$ map $\mu_n\to\mu_n$.  Hence this composition is an isomorphism. Composing with the isomorphism at the bottom of \eqref{eq:bigdiagram}, it follows that $\chi_i\cup \chi_i'$, and more generally $\sum_{i=1}^d [\chi_i\cup \chi_i']$, can be identified with its image in $H^2(C\otimes_K\overline{K},\mu_n^{\otimes 2})$.

Finally, by \cite[Chap. V, Remark 2.4 (f)]{Milne80}, the following diagram commutes
$$
\begin{CD}
H^1(C\otimes_K\overline{K},\mu_n) @.\times H^1(C\otimes_K\overline{K},\mu_n) @>\cup>> H^2(C\otimes_K\overline{K},\mu_n^{\otimes 2}) \\
@| @| @| \\
\Pic^0(C\otimes_K\overline{K})[n] @.\times \Pic^0(C\otimes_K\overline{K})[n] @>e_n>> \mu_n
\end{CD}
$$
We deduce that the image of $\sum_{i=1}^d [\chi_i\cup \chi_i']$ in $H^2(C\otimes_K\overline{K},\mu_n^{\otimes 2})=\mu_n$ can be identified with
$$
\prod_{i=1}^d e_n(L_i,L_i')
$$
which is trivial by hypothesis. According to Theorem~\ref{thm:Sharifi}, there exists a $\Heis_{2d+1}(\mu_n)$-torsor $C'\to C$ whose associated $(\mu_n)^{2d}$-torsor is the $2d$-tuple $(\chi_1,\dots,\chi_d,\chi_1',\dots, \chi_d')$. By Lemma~\ref{lem:centraltwist}, we can make a constant field twist of the central action of the twisted Heisenberg group on $C'$ to ensure that $C' \to C$ splits over $P_0$. This completes the proof of part (2).

For part (3), we use that by Kummer theory, we have an isomorphism
$$
H^1(C\otimes_K \overline{K},(\mu_n)^{2d}) \simeq \Hom((\Z/n\Z)^{2d},\Pic^0(C\otimes_K \overline{K}))
$$
under which connected torsors correspond to injective morphisms. But the image of the $(\mu_n)^{2d}$-torsor $(\chi_1,\dots,\chi_d,\chi_1',\dots, \chi_d')$ is none other than the map defined by the $2d$-tuple $(L_1,\dots,L_d,L_1',\dots, L_d')$. Therefore, our $(\mu_n)^{2d}$-torsor is geometrically connected  if and only if the subgroup generated by the $L_i$ and the $L_i'$ in $\Pic^0(C\otimes_K \overline{K})$ is isomorphic to $(\Z/n\Z)^{2d}$. The map $\Pic^0(C)\to \Pic^0(C\otimes_K \overline{K})$ being injective, it suffices to check this over $K$. The conclusion of part (3) follows from the last statement in Theorem~\ref{thm:Sharifi}.
\end{proof}


\subsection{An example with $d=1$ over $\Q$}

\begin{corollary}
\label{cor:H1}
Let $\lambda\in \Q^{\times}$, $\lambda^2 \neq 1$, and let $n>1$ be an odd integer. Let $C$ be the hyperelliptic curve defined over $\Q$ by the affine equation $y^2=x^{2n}-(1+\lambda^2)x^n+\lambda^2$, and let $K$ be a number field such that $[K(\mu_n):K]$ is prime to $n$. Then there exists a geometrically connected $\Heis_3(\mu_n)$-torsor over $C\otimes_\Q K$, which splits over the point $P_0=(1,0)$.
\end{corollary}

We note that, if $n$ is prime, or more generally if $\varphi(n)$ is prime to $n$, then the hypotheses of Corollary~\ref{cor:H1} are satisfied for $K=\Q$.

\begin{proof}
We note that $P_0$ is a rational Weierstrass point of $C$. It is proved in \cite[Lemma~3.3]{GL12} that $\Pic^0(C)$ contains two independent classes of order $n$, which we denote by $L$ and $L'$. The Weil pairing $e_n(L,L')$ takes values in $\mu_n(\Q)=\{1\}$, therefore $e_n(L,L')=1$. If we consider the classes $L$ and $L'$ in $\Pic^0(C\otimes_\Q K)$, then their Weil pairing over $K$ has the same value, hence the assumptions of Theorem~\ref{thm:main} are satisfied, and the result follows.
\end{proof}


\subsection{An example with $d=2$, $n=3$ over $\Q$}

\begin{corollary}
\label{cor:H2}
There exists a hyperelliptic curve $C$ defined over $\Q$, together with a rational Weierstrass point $P_0$, and a geometrically connected $\Heis_5(\mu_3)$-torsor over $C$ which splits over $P_0$.
\end{corollary}

\begin{proof}
Following an idea of Craig,  a construction is given in  \cite[Theorem~2.3]{GLsurvey} of a hyperelliptic curve $C$ over $\Q$, together with a rational Weierstrass point $P_0$, and four independent classes in $\Pic(C)[3]$. The same argument as in Corollary~\ref{cor:H1} proves that the assumptions of Theorem~\ref{thm:main} are satisfied, hence the result.
\end{proof}


\subsection{A remark on the general case}
\label{rmk:general}

Let $F$ be a number field, and let $C$ be a hyperelliptic curve of genus $g$ over $F$, with an $F$-rational Weierstrass point $P_0$. Let $K$ be the field of definition of the points of $J_C[n]$, the full $n$-torsion subgroup of the Jacobian of $C$. Then $K(\mu_n)=K$, because the Weil pairing is non-degenerate. It is easy to check that the hypotheses of Theorem~\ref{thm:main} are satisfied for the curve $C$ over $K$ when putting $d=g$.  Hence there exists a geometrically connected, {\'e}tale $\Heis_{2g+1}(\Z/n\Z)$-torsor over $C$ which splits over $P_0$, and whose associated $(\Z/n\Z)^{2g}$-torsor is the maximal (pointed) {\'e}tale Galois cover of $C$ whose Galois group is an $n$-torsion abelian group.

Given $F$ and $n$, one may ask which number fields $K$ can be obtained as the field of definition of the points of $J_C[n]$ for some hyperelliptic curve $C$ of genus $g$. We note that, in any case, $K/F$ is Galois and
$$
\Gal(K/F) \hookrightarrow \GSp_{2g}(\Z/n\Z)
$$
where $\GSp$ denotes the general symplectic group. Indeed, $J_C[n]$ is a Galois module with underlying abelian group isomorphic to $(\Z/n\Z)^{2g}$, and the Weil pairing on $J_C[n]$ is non-degenerate and alternating.  It follows that $\Gal(\overline{K}/K)$ acts on $J_C[n]$ via $\GSp_{2g}(\Z/n\Z)$, so 
$[K:F]$ divides $\# \GSp_{2g}(\Z/n\Z)$.  Sharper bounds on $[K:F]$ can be obtained with more hypotheses on $C$, e.g. by assuming that 
$J_C$ has complex multiplication.

In view of such constructions involving torsion points on hyperelliptic curves, the following question naturally arises (see also \cite[Question~3.5]{GL12}):

\begin{question}
\label{quest:Jacobian}
Given positive integers $n$ and $r$, does there exist a hyperelliptic curve $C$ defined over $\Q$ such that $\Pic^0(C)$ contains a subgroup isomorphic to $(\Z/n\Z)^r$?
\end{question}

The curve defined in Corollary~\ref{cor:H1} gives a positive answer to this question when $r=2$ and $n$ is arbitrary. To our knowledge, this is currently the strongest known general result concerning this question. Solutions for specific small pairs $(n,r)$ are given in \cite{HLP00}.

In view of the previous discussion, one may of course replace $\Q$ by an arbitrary number field $F$.


\section{Arithmetic specialization}
\label{s:Arithmetic}

Throughout this section, $K$ denotes a number field, and $S$ a finite set of places of $K$. We denote by $\mathcal{O}_{K,S}$ the ring of $S$-integers of $K$, obtained by inverting in the full ring of integers of $K$ all finite places which belong to $S$.

Let us consider a finite {\'e}tale (not necessarily commutative) $\mathcal{O}_{K,S}$-group scheme $G$. We denote by $H^1_{\et}(\mathcal{O}_{K,S},G)$  the cohomology set which classifies {\'e}tale $G$-torsors over $\mathcal{O}_{K,S}$. We denote by $G_K$ the generic fiber of $G$, and by $H^1(K,G_K)$ the (possibly non-abelian) Galois cohomology set $H^1(\Gal(\bar{K}/K),G_K(\bar{K}))$.

Let us recall that the ``restriction to the generic fiber'' map $H^1_{\et}(\mathcal{O}_{K,S},G)\to H^1(K,G_K)$ is injective. This allows us to identify $H^1_{\et}(\mathcal{O}_{K,S},G)$ with a subset of $H^1(K,G_K)$.

We now define a set of cohomology classes which are locally trivial at all places in $S$.  These classes will be called $S$-split.

\begin{definition}
\label{S_split_def}
If $S$ is a finite set of places of $K$, we let
$$
H^1_{\text{$S$-split}}(\mathcal{O}_{K,S},G) := \ker\left(H^1_{\et}(\mathcal{O}_{K,S},G) \to \prod_{v\in S} H^1_{\et}(K_v,G_{K_v})\right).
$$
\end{definition}

In more algebraic terms, $H^1_{\text{$S$-split}}(\mathcal{O}_{K,S},G)$ is the subset of $H^1(K,G_K)$ consisting of $K$-algebras which are unramified at finite places outside $S$, and in which all places in $S$ (including the infinite ones) are totally split. In particular, such algebras are unramified at all finite places of $K$.

We are now ready to state our specialization theorem, which follows immediately from the results of Bilu and the third author \cite{bg18}.

\begin{theorem}
\label{thm:specialization}
Let us consider the setting of Theorem~\ref{thm:main}, with the additional assumption that $K$ is a number field. Let $\psi:\tilde{C}\to C$ be a geometrically connected {\'e}tale $\Heis_{2d+1}(\mu_n)$-torsor which splits over the point $P_0$, whose existence is ensured by Theorem~\ref{thm:main}. Then there exists a finite set $S$ of places of $K$ with the following properties.
\begin{enumerate}
\item $S$ contains all places above $n$;
\item the torsor $\psi:\tilde{C}\to C$ extends to a $\Heis_{2d+1}(\mu_n)$-torsor between smooth $\mathcal{O}_{K,S}$-curves.
\end{enumerate}
Moreover, given any such $S$, there exist infinitely many (isomorphism classes of) quadratic extensions $L/K$ with the following properties.  There is a point $P\in C(L)$ such that the specialization of $\psi$ at $P$ is a connected $\Heis_{2d+1}(\mu_n)$-torsor and belongs to the subset $H^1_{\text{$S$-split}}(\mathcal{O}_{L,S},\Heis_{2d+1}(\mu_n))$.

Furthermore, there is a constant $c>0$ depending only on $K$, $\psi$ and $S$ for which the following is true.  Let $g(C)$ denote the genus of $C$.  
For sufficiently large positive $N$, the number of (isomorphism classes of) such fields $L$ whose relative discriminant $\Delta(L/K)$ satisfies
$$
\bigl|\mathrm{N}_{K/\Q}\Delta(L/K)\bigr|^{1/[K:\Q]} \leq N
$$
is at least $cN^{[K:\Q]/(4g(C)+2)}/\log N$.
\end{theorem}

In the statement above, slightly abusing notation, we denote by the same letter $S$ the set of places of $L$ lying above places in $S$.

Here is another way to state the conclusion of Theorem~\ref{thm:specialization}.  For infinitely many quadratic extensions $L/K$, we obtain by specializing $\psi$ an extension of degree $n^{2d+1}$ of $L$ that is a $\Heis_{2d+1}(\mu_n)$-torsor, in which all finite places are unramified and places above $S$ are totally split.

In the case where $K(\mu_n)=K$, $\Heis_{2d+1}(\mu_n)$ is isomorphic to the constant Heisenberg group scheme $\Heis_{2d+1}(\Z/n\Z)$, and we obtain by specializing $\psi$ a Galois extension of $L$ with group $\Heis_{2d+1}(\Z/n\Z)$, in which all finite places are unramified and places above $S$ are totally split.

\begin{remark}
Given a set $S$ satisfying (1) and (2) in Theorem~\ref{thm:specialization}, any larger set also satisfies these conditions. When enlarging $S$, the condition that the specialization of $\psi$ at $P$ belongs to $H^1_{\text{$S$-split}}(\mathcal{O}_{L,S},\Heis_{2d+1}(\mu_n))$ is stronger, which has the effect of enlarging the constant $c$ in the quantitative statement.
\end{remark}

\begin{proof}[Proof of Theorem~\ref{thm:specialization}]
The existence of a set $S$ satisfying conditions (1) and (2) follows by elementary considerations on the reduction of covers of curves, known as the Chevalley-Weil theorem; see for example \cite[Section~4.2]{Se97}.  A hyperelliptic curve with a rational Weierstrass point admits an affine model of the form $y^2=f(x)$ where $f\in K[x]$ is a polynomial of degree $2g+1$.  Theorem~\ref{thm:specialization} now follows from  \cite[Theorems~4.3 and 4.7]{bg18}.
\end{proof}


\subsection{Unramified twisted Heisenberg torsors over quadratic fields}

By combining Corollary~\ref{cor:H1} and Theorem~\ref{thm:specialization}, we obtain the following result. 

\begin{corollary}
Let $p$ be an odd prime. Then there exist infinitely many (imaginary and real) quadratic fields $L$ which admit a connected $\Heis_3(\mu_p)$-torsor, unramified at all finite places of $L$. Moreover, given a finite set $T$ of prime numbers containing $p$, there exist infinitely many such $L$ for which primes lying above $T$ in $L$ are totally split in the extension corresponding to the $\Heis_3(\mu_p)$-torsor.
\end{corollary}

We assume here that $p$ is prime in order to ensure that $[\Q(\mu_p):\Q]$ is prime to $p$, which is required in the statement of Corollary~\ref{cor:H1}. When applying Theorem~\ref{thm:specialization}, we choose a sufficiently large set $S$ containing $T$ and satisfying the required conditions.

\medbreak
One can also deduce from Theorem~\ref{thm:specialization} a quantitative version of the statement above.


\subsection{Unramified Heisenberg Galois extensions over quadratic extensions of cyclotomic fields}

By combining Corollary~\ref{cor:H1} and Theorem~\ref{thm:specialization} over the cyclotomic field $\Q(\zeta_n)$, one obtains Theorem~\ref{thm:intro} stated in the introduction.


\subsection{General remarks about the unramified Inverse Galois problem}

It is a folklore conjecture that every finite group occurs as the Galois group of an unramified Galois extension of some quadratic number field.
 In the case of finite abelian groups, this would be an immediate consequence of the Cohen-Lenstra heuristics. 
 
In fact, in the abelian case, the strongest general result concerning this folklore conjecture is obtained by arithmetic specialization from the hyperelliptic curve of Corollary~\ref{cor:H1}. More precisely, one obtains from this curve infinitely many imaginary quadratic fields whose class group contains a subgroup isomorphic to $(\Z/n\Z)^2$, and the rest follows from class field theory.

The construction of unramified Galois extensions of small degree number fields has a long history. Shafarevich proved that the Inverse Galois problem over $\Q$ has a solution for solvable groups. Using Shafarevich-type methods, it was recently proved by Kim \cite{Kim19} that, if $G$ is a solvable group of exponent $g$, there exists a number field $K$ of degree $g$ such that $G$ can be realized as the Galois group of an everywhere unramified extension of $K$. Since $\Heis_{2d+1}(\Z/n\Z)$ is a metabelian group of exponent $n$, one deduces from Kim's result that there exist number fields $K$ of degree $n$ over which $\Heis_{2d+1}(\Z/n\Z)$ can be realized as the Galois group of an everywhere unramified extension.  However, the approach in \cite{Kim19} by itself does not apply to the folklore conjecture as soon as $G$ has exponent larger than $2$.   

Our Theorem~\ref{thm:intro} is a result in the direction of the folklore conjecture, using a completely different, geometric approach. The advantage of this approach is that it can in principle lead to  proofs of the folklore conjecture for various $G$ of large exponent provided one can construct hyperelliptic curves with suitable properties.  
More explicitly, the proof of Theorem~\ref{thm:intro} suggests a connection between Question~\ref{quest:Jacobian} and the folklore conjecture.


\end{document}